\documentclass[12pt]{article}
\usepackage{amssymb}
\usepackage{color}
\usepackage[dvipsnames]{xcolor}
\usepackage{amsmath}
\usepackage{cite}
\usepackage{enumerate}
\newenvironment{proof}[1][Proof]{\noindent\textit{#1.} }{\hfill  \rule{0.5em}{0.5em}}

\newtheorem{theorem}{Theorem}
\newtheorem{lm}{Lemma}
\newtheorem{thm}[theorem]{Theorem}
\newtheorem{pro}{Proposition}
\newtheorem{exmp}{Example}

\newtheorem{coro}{Corollary}

\begin{document}

\title{An equation in nonlinear combination of iterates
}

\author{
Chaitanya Gopalakrishna\,$^a$,~~
Weinian Zhang\,$^b$
\vspace{3mm}\\
$^a${\small Theoretical Statistics and Mathematics Unit, Indian Statistical Institute,}
\\
{\small  R.V. College Post, Bengaluru-560059, India}
\\
$^b${\small School of Mathematics, Sichuan University,}
\\
{\small Chengdu, Sichuan 610064, P. R. China}
    \vspace{0.2cm}\\
    {\small cberbalaje@gmail.com (CG),~~matzwn@126.com (WZ).}}

\date{}


\maketitle

\begin{abstract}
In this paper we deal with an equation in nonlinear combination of iterates.
Although it can be reduced by the logarithm conjugacy to a form for application of Schauder's or Banach's fixed point theorems,
a difficulty called {\it Zero Problem} is encountered for continuous solutions
because the domain does not contain $0$.
So we consider solutions with weaker regularity, using the Knaster-Tarski fixed point theorem for complete lattices to give order-preserving solutions. Then we give semi-continuous solutions and integrable solutions.

\vskip 0.2cm

{\bf Keywords:}
Functional equation; iteration; complete lattice; order-preserving map; Knaster-Tarski fixed point theorem.
\vskip 0.2cm

{\bf MSC(2010):}
primary 39B12; secondary 47J05; 06F20.
\end{abstract}



\baselineskip 16pt
\parskip 10pt


\section{Introduction}

Iteration, the most popular action in the contemporary era because of computers (\cite{Ian-Kinder, Kinder, Martin2002, Wynn}),
can be understood as
repetition of the same operation.
In mathematics,
for integer $n\ge 0$, the $n$-th order iterate $g^n$ of
a self-map $g: E\to E$ on a non-empty set $E$ is defined recursively by
$$
g^n = g\circ g^{n-1},\ \ \  g^0={\rm id},
$$
where $\circ$ denotes the composition of functions and ${\rm id}$
 the identity map.
Great attentions (\cite{Bar-Jar, Kucz68,Kucz90}) have been paid to
functional equations involving iteration, called {\it iterative equations}.
The general form of such equations can be presented as
\begin{eqnarray}
\Phi(x, g(x), g^2(x),\ldots, g^n(x))= 0,
\label{Eq*}
\end{eqnarray}
where $\Phi$ is a given function and $g$ is unknown.

There have been published many papers on equation (\ref{Eq*}) with specific $\Phi$.
The simplest one is $g^n = G$, i.e., $\Phi(x_0,\ldots,x_n)=x_n-G(x_0)$, where $G$ is given,
which is called the iterative root problem (\cite{Kucz68,Kucz90,Targonski}).
As shown in \cite{Kucz90},
the problem of invariant curve $y = g(x)$ for the planar mapping $(x,y)\to (y, F(x,y))$
can be reduced to the equation $F(x, g(x)) = g^2(x)$, which is equation (\ref{Eq*}) with
$\Phi(x_0,x_1) = F(x_0,x_1)- x_2$.
When $\Phi$ is of the linear combination form
$\Phi(x_0,\ldots,x_n)= 
\sum_{k=1}^{n}\lambda_kx_k- G(x_0)$,
equation (\ref{Eq*}) becomes the following
\begin{eqnarray}
\sum_{k=1}^{n}\lambda_kg^k(x)=G(x),
\label{Poly-Eq}
\end{eqnarray}
called the polynomial-like iterative equation, which was investigated in various aspects
such as continuous solutions (\cite{Jar96,Kul-Tab,Tabor,Xu-Zhang}), differentiable solutions (\cite{zhang1990}),
convex solutions
and decreasing solutions (\cite{Bing-Weinian}), and equivariant solutions (\cite{Zhang-Edinb}).

More difficulties come from the nonlinear case of $\Phi$.
In \cite{Mu-Su,Si,Wang-Si}
 the function $\Phi$ is nonlinear but Lipschitzian, which makes $\Phi$ being bounded by a linear combination
and therefore the method for equation (\ref{Poly-Eq}) is available.
Another discussion (\cite{Zdun-Zhang}) on nonlinear $\Phi$ was made on the unit circle $\mathbb{T}^1$, which was
solved by lifting to maps on a compact interval.
Thus, more attentions are paid to the forms which are not Lipschtzian, for example,
$
 \Phi(x_0,x_1,x_2):= \alpha_1x_1^{\lambda_1}x_2^{\mu_1} +\alpha_2x_1^{\lambda_2}x_2^{\mu_2}.
$
In \cite{Gop-Vee} the equation
\begin{eqnarray}
\prod_{k=1}^{n}(g^k(x))^{\lambda_k}
 = G(x),
\label{ggg}
\end{eqnarray}
the equation (\ref{Eq*}) with multiplication, i.e.,
$\Phi(x_0,\ldots, x_n) =\prod_{k=1}^n x_k^{\lambda_k}-G(x_0)$, is discussed,
where $G$ is given, $\lambda_k$\,s are real constants
and $g$ is unknown.
This equation can be reduced with a logarithm conjugation to the standard form of the
polynomial-like iterative equation (\ref{Poly-Eq}) on the whole $\mathbb{R}$,
but those known results about (\ref{Poly-Eq}) were obtained on a compact interval or a neighborhood of
 a fixed point.
The authors used the Banach contraction principle to give the existence, uniqueness,
and continuous dependence of continuous solutions on $\mathbb{R}_+:=(0,\infty)$ that are Lipschitzian on their ranges and constructed its
continuous solutions on $\mathbb{R}_+$ sewing piece by piece. Then they technically extended the
results on $\mathbb{R}_+$ to $\mathbb{R}_-:=(-\infty,0)$ and showed that none of the pairs of solutions obtained on $\mathbb{R}_+$
and $\mathbb{R}_-$ can be combined at the origin to get a continuous solution of the equation
on the whole $\mathbb{R}$, but can extend those given on $\mathbb{R}_+$ to obtain continuous solutions
on the whole $\mathbb{R}$.
A discussion on differentiable solutions of \eqref{ggg} on $\mathbb{R}_+$ and their extension to $\mathbb{R}_-$ was made in \cite{Gopal2023}.

In this paper we investigate the more general equation
\begin{eqnarray}\label{prod}
\prod_{k=1}^{n}\Psi_k(g^k(\psi_k(x)))=G(x)
\end{eqnarray}
 on a compact interval $[c,d]$, where $G$, $\Psi_k$ and $\psi_k$ are given for $1\le k \le n$
  and $g$ is unknown.
In the particular case that
$\Psi_k(x)=x^{\lambda_k}$ and $\psi_k(x)=x$,
this equation is exactly the same as (\ref{ggg}),
but in general this equation cannot be
 reduced to a linear combination of iterates by the logarithm conjugacy
because not only nontrivial $\Psi_k, \psi_k$ may be included in (\ref{prod}) but also the domain $[c,d]$ may contain $0$.
In Section \ref{sec-2} we use logarithm conjugacy to reduce the product in (\ref{prod}) to a sum, which allows us to apply the
Banach Contraction Principle and the Schauder fixed point theorem  to prove the existence and uniqueness of continuous solutions.
However, as we remark at the end of the section,
those continuous solutions do not have a domain containing $0$,
referring to the {\it Zero Problem}.
In order to deal with this problem, in Sections \ref{S2} and \ref{S4}  we consider solutions with weaker regularity.
As we require monotonicity in Section \ref{sec-2},  we first give order-preserving solutions in Section \ref{S2} using
 the Knaster-Tarski fixed point theorem for complete lattices.
Then we discuss  semi-continuous solutions in Section  \ref{S4}.
Finally, in Section \ref{sec6} we
illustrate our results with examples and make remarks on integrable solutions.
We leave some problems for future discussion.


\section{Continuous solutions
 }\label{sec-2}

In this section we give results on the existence, uniqueness  and stability of continuous solutions of \eqref{prod} on $J=[c,d]$, where $d>c>0$.
For each compact interval $I:=[a,b]$ in $\mathbb{R}$ with $a<b$,
let
$\mathcal{C}(I,\mathbb{R})$ (resp. $\mathcal{C}(I,\mathbb{R}_+)$, resp. $\mathcal{C}(I,\mathbb{R}_-)$, resp. $\mathcal{C}(I,I)$) be the set of all continuous maps on $I$ into $\mathbb{R}$ (resp. into $\mathbb{R}_+$, resp. into $\mathbb{R}_-$, resp. into $I$).
Then $\mathcal{C}(I,\mathbb{R})$ is a Banach space
in the uniform norm $\|\cdot\|_I$ defined by
$\|f\|_I=\sup \{|f(x)|: x\in I\}$. Considering $g, G, \psi_k\in \mathcal{C}(J,J)$ and $\Psi_k\in \mathcal{C}(J,\mathbb{R}_+)$ for $1\le k\le n$,
we can use the logarithmic map $x\mapsto \log x$
to conjugate them to maps on $I$
and reduce the product in equation (\ref{prod}) on $J$ to a sum to obtain the equation
\begin{eqnarray}\label{sum}
	\sum_{k=1}^{n}\Phi_k(f^k(\phi_k(x)))=F(x)
\end{eqnarray}
on $I$, where $I=\log J:=\{\log x:x\in J\}$, and $f(x)=\log g(e^x)$, $F(x)=\log G(e^x)$, $\Phi_k(x)=\log \Psi_k(e^x)$ and $\phi_k(x)=\log\psi_k(e^x)$  for all $x\in I$ and $1\le k \le n$. More precisely, we have the following.

\begin{pro}\label{P1}
	Let
	$G \in \mathcal{X}\subseteq  \mathcal{C}(J,J)$. Then
	a map $g$ is a solution (resp. unique solution) of \eqref{prod} in $\mathcal{X}'\subseteq \mathcal{C}(J,J)$, where
	 $\Psi_k \in \mathcal{Y}_k\subseteq  \mathcal{C}(J,\mathbb{R}_+)$
	 and $\psi_k \in \mathcal{Z}_k\subseteq  \mathcal{C}(J,J)$ for all $1\le k\le n$,
	if and only if $f=h^{-1}\circ g\circ h$
	is a solution (resp. unique solution) of \eqref{sum} in $\mathcal{\tilde{X}'}\subseteq \mathcal{C}(I,I)$, where $h(x)=e^x$, $I=\log J$, $F=h^{-1}\circ G\circ h\in \tilde{\mathcal{X}}$,
	$\mathcal{\tilde{X}}=\{h^{-1}\circ g\circ h: g\in \mathcal{X}\}\subseteq \mathcal{C}(I,I)$, $\mathcal{\tilde{X}'}=\{h^{-1}\circ g\circ h: g\in \mathcal{X'}\}\subseteq \mathcal{C}(I,I)$,   and $\Phi_k=h^{-1}\circ \Psi_k\circ h \in \mathcal{\tilde{Y}}_k$,   $\phi_k=h^{-1}\circ \psi_k\circ h \in \mathcal{\tilde{Z}}_k$, $\mathcal{\tilde{Y}}_k=\{h^{-1}\circ g\circ h: g\in \mathcal{Y}_k\}\subseteq \mathcal{C}(I,\mathbb{R})$ and $\mathcal{\tilde{Z}}_k=\{h^{-1}\circ g\circ h: g\in \mathcal{Z}_k\}\subseteq \mathcal{C}(I,I)$  for all $1\le k\le n$.
\end{pro}

\begin{proof}
	Let $g$ be a solution of \eqref{prod} in $\mathcal{X}'$. Since $h$ is a homeomorphism of $I$ onto $J$, clearly $\mathcal{\tilde{X}'} \subseteq \mathcal{C}(I,I)$ and $f\in \mathcal{\tilde{X}'}$.  Also, for each $x\in I$, we have
	\begin{align*}
		\sum_{k=1}^{n}\Phi_k(f^k(\phi_k(x)))&=\sum_{k=1}^{n} \log\Psi_k (g^k(\psi_k(e^x)))\\
		&= \log\left(\prod_{k=1}^{n}\Psi_k(g^k(\psi_k(e^x)))\right)
		=\log G(e^x)=F(x),
	\end{align*}
	implying that $f$ is a solution of \eqref{sum} on $I$.  The converse follows similarly. Next, in order to prove the uniqueness, assume that \eqref{prod} has a unique solution in $\mathcal{X'}$ and
	suppose that $f_1, f_2$ are any two solutions of \eqref{sum} in $\mathcal{\tilde{X}'}$. Then, by the ``if'' part of what we have proved above, there exist solutions  $g_1$ and $g_2$ of \eqref{prod} in $\mathcal{X'}$ such that $f_1=h^{-1}\circ g_1\circ h$ and $f_2=h^{-1}\circ g_2\circ h$. By our assumption, we have $g_1=g_2$ and therefore $f_1=f_2$. The proof of the converse is similar.
\end{proof}

By Proposition \ref{P1}, it suffices to prove the existence and uniqueness of continuous solutions for \eqref{sum} on $I=\log J$
in order to prove those for \eqref{prod} on $J$. For compact intervals $J$ and $I$, as defined at the beginning of the section, and $M, \delta\ge 0$, let
\begin{align*}
	\mathcal{G}(J;\delta,M)&:=\left\{g\in \mathcal{C}(J,J):~g(c)=c, g(d)=d\right. \text{and}\\
	& ~~~~\left.(x/y)^\delta \le g(x)/g(y)\le (x/y)^M, \forall x, y \in J~\mbox{with}~x\ge y\right\}\\
	\mathcal{F}(I;\delta,M)&:=\{f\in \mathcal{C}(I,I):~f(a)=a, f(b)=b~\text{and}\\
	& ~~~~~\delta(x-y)\le f(x)-f(y)\le M(x-y), \forall x, y \in I~\mbox{with}~x\ge y\}.
\end{align*}
Then an easy verification shows that
\begin{align}\label{G-F}
	g\in 	\mathcal{G}(J;\delta,M)\quad\text{if and only if}\quad h^{-1}\circ g\circ h  \in \mathcal{F}(I;\delta,M)	
\end{align}
for all $M,\delta \ge 0$, where $h(x)=e^x$ and $I=\log J$. Furthermore,  Proposition 2.1 in \cite{Mu-Su} and \eqref{G-F} shows that
\begin{align*}
	\mathcal{G}(J;\delta,M)&=\left\{\begin{array}{cll}
		\emptyset &\text{if}&M<1~\text{or}~\delta>1,\\
		\{{\rm id}\}&\text{if}&M=1~\text{or}~\delta=1,
	\end{array}\right.
\end{align*}
indicating that we cannot seek solutions of \eqref{prod} in $\mathcal{G}(J;\delta,M)$ without restricting $M$ and $\delta$.

We now state two lemmas
on the existence, uniqueness and stability of solutions of continuous solutions for \eqref{sum} on $I$. They
 are given in Theorems 3.1 and 3.4 in \cite{Mu-Su} on the interval $[0,1]$, but
the generalization to $[a,b]$ is trivial.

\begin{lm}
\label{L1}
		Let $0<\lambda_1<1$ and $\lambda_k \ge 0$ for $2\le k\le n$
	such that $\sum_{k=1}^{n}\lambda
	_k=1$,
	and	$\phi_k={\rm id}$ and
	 $\Phi_k=\lambda_k\Upsilon_k$ with $\Upsilon_k\in \mathcal{F}(I;l_k,L_k)$ such that  $L_k\ge  l_k\ge 0$ for all $1\le k \le n$. If
	$F\in \mathcal{F}(I;K_1\delta, K_0M)$, then \eqref{sum} has a solution $f$ in $\mathcal{F}(I;\delta, M)$, where $0<\delta< 1< M$, $K_0=\sum_{k=1}^{n}\lambda_kl_k\delta^{k-1}$  and $K_1=\sum_{k=1}^{n}\lambda_kL_kM^{k-1}$.
\end{lm}

\begin{lm}
\label{L2}
In addition to the hypotheses of Lemma \ref{L1}, suppose that
\begin{eqnarray}
\label{KKK}
K:=\lambda_1l_1-\sum_{k=2}^{n}\lambda_k\left(L_k\frac{M^{k-1}-1}{M-1}-l_k\delta^{k-1}\right)>0.
\end{eqnarray}
Then for each $F\in  \mathcal{F}(I;K_1\delta, K_0M)$,  \eqref{sum} has a unique solution $f$ in $\mathcal{F}(I;\delta, M)$. Furthermore,  if  $F_1\in  \mathcal{F}(I;K_1\delta, K_0M)$ and $f_1\in \mathcal{F}(I;\delta, M)$ satisfies
		$	\sum_{k=1}^{n}\Phi_k(f^k(\phi_k(x)))=F_1(x)$ for all $x\in I$, then
	\begin{eqnarray}\label{01}
		\|f-f_1\|_I\le \frac{1}{K}\|F-F_1\|_I,
	\end{eqnarray}
i.e., the solution $f$ continuously depends on $F$.
\end{lm}

Since $M>1>\delta>0$ and $L_k\ge l_k\ge 0$
 for all $1\le k\le n$, the condition (\ref{KKK}) requires
$\lambda_1$ to be large in comparison with other $\lambda_k$'s. This condition will be checked in Example \ref{Exmp1}.

\begin{thm}\label{C-existance}
	Let $0<\lambda_1<1$ and $\lambda_k \ge 0$ for $2\le k\le n$
	such that $\sum_{k=1}^{n}\lambda_k=1$,
	and let	$\psi_k={\rm id}$ and
	 $\Psi_k(\cdot)=(\Xi_k(\cdot))^{\lambda_k}$ with $\Xi_k\in \mathcal{G}(J;l_k,L_k)$ such that  $L_k\ge l_k\ge 0$ for all $1\le k \le n$.
	 Suppose further that (\ref{KKK}) is true.
Then
\eqref{prod} has a unique solution $g$ in $\mathcal{G}(J;\delta, M)$, which depends on $G$  continuously,
if
$G\in \mathcal{G}(J;K_1\delta, K_0M)$, where
$0<\delta< 1< M$, $K_0=\sum_{k=1}^{n}\lambda_kl_k\delta^{k-1}$ and
 $K_1=\sum_{k=1}^{n}\lambda_kL_kM^{k-1}$.
\end{thm}

\begin{proof}
Let $G\in \mathcal{G}(J;K_1\delta, K_0M)$, $a:=\log c$ and $b:=\log d$. Then we obtain the interval $I=[a,b]$ with $a<b$,
which satisfies $I=\log J$.
Further, we have $F:=h^{-1}\circ G\circ h \in  \mathcal{F}(I;K_1\delta, K_0M)$ and $\phi_k:=h^{-1}\circ \psi_k\circ h={\rm id}$
for all $1\le k\le n$, where  $h (x)=e^x$. Moreover  $\Phi_k:=h^{-1}\circ \Psi_k\circ h$ and $\Upsilon_k:=h^{-1}\circ \Xi_k\circ h$ satisfy $\Phi_k=\lambda_k\Upsilon_k$ and $\Upsilon_k\in  \mathcal{F}(I;l_k,L_k)$ for all $1\le k\le n$. Therefore,
as $K>0$ assumed in (\ref{KKK}),
by Lemmas \ref{L1} and \ref{L2} we see that \eqref{sum} has a unique solution $f$ in $\mathcal{F}(I;\delta, M)$.	
This implies by Proposition \ref{P1}  that $g:=h \circ f\circ h^{-1}$ is the unique solution of \eqref{prod} in $\mathcal{G}(J;\delta, M)$.

	
	Next, in order to prove the continuous dependency of $g$ on $G$, suppose that 	$G_1\in \mathcal{G}(J;K_1\delta, K_0M)$ and $g_1 \in \mathcal{G}(J;\delta, M)$ satisfy that
	\begin{align*}
		\prod_{k=1}^{n}\Psi_k(g_1^{k}(\psi_k(x)))&=G_1(x)
	\end{align*}
	on $J$. Let   $F_1:=h^{-1} \circ G_1\circ h$ and $f_1:=h^{-1} \circ g_1\circ h$.
	Since $G_1 \in \mathcal{G}(J;K_1\delta, K_0M)$,  we have $F_1\in \mathcal{F}(I;K_1\delta, K_0M)$.
	Similarly, we see that $f_1\in \mathcal{F}(I;\delta, M)$.
	Further,  $f_1$ satisfies
	\begin{align*}
		\sum_{k=1}^{n} \Phi_k(f_1^k(\phi_k(x)))&=F_1(x)
	\end{align*}
	on $I$. Moreover, by Lemma \ref{L2} we know that  \eqref{01} is satisfied.	
	Since the map $x\mapsto e^x$ is continuously differentiable on $I$ with bounded derivative, it is a Lipschitzian
	map on $I$. In fact,
$		|e^x-e^y|<e^b|x-y|$ for all $x, y\in I.$
	So, for each $x\in J$, we have
	\begin{eqnarray*}
		|g(x)-g_1(x)|=|e^{f(\log x)}-e^{f_1(\log x)}|
		<e^b|f(\log x)-f_1(\log x)|
		\le  e^b\|f-f_1\|_I,
	\end{eqnarray*}
	implying that
	\begin{align}\label{03}
		\|g-g_1\|_J&\le e^b \|f-f_1\|_I
\le  \frac{d}{K} \|F-F_1\|_I~~(\text{using} ~\eqref{01}).
	\end{align}
Since the map $x\mapsto \log x$ is continuously differentiable on $J$ with bounded derivative, it is a Lipschitzian
	map on $J$. In fact,
$		|\log x-\log y|<\frac{1}{c}|x-y|$ for all $x, y\in J.$
	Therefore, for each $x\in I$, we have
	\begin{eqnarray*}
		|F(x)-F_1(x)|=|\log G(e^x)-\log G_1(e^x)|
		< \frac{1}{c}|G(e^x)-G_1(e^x)|
		\le \frac{1}{c}\|G-G_1\|_J,
	\end{eqnarray*}
	implying that
	\begin{eqnarray}\label{04}
		\|F-F_1\|_I\le \frac{1}{c} \|G-G_1\|_J.
	\end{eqnarray}
	Then,  from \eqref{03} and \eqref{04} we have
	\begin{eqnarray*}
		\|g-g_1\|_J\le \frac{d}{cK}\|G-G_1\|_J.
	\end{eqnarray*}
This completes the proof.
\end{proof}

The assumptions that $0< \lambda_1 <1$ and $\sum_{k=1}^{n}\lambda_k=1$, made in Theorem \ref{C-existance}, is not strong.
In fact, if $\lambda_1>1$
or $\sum_{k=1}^{n}\lambda_k>1$, then we can divide
all the exponents $\lambda_k$s in \eqref{prod}
 by $\sum_{k=1}^{n}\lambda_k$ to get the normalized equation,
but the assumptions on $G$ have to be modified suitably.

Further, although Theorem \ref{C-existance} is given for $J$ such that $d>c>0$, we can use it, together with an idea of conjugation through the
reflection map, to give a similar result on continuous solutions of \eqref{prod} when $c<d<0$. More precisely, the following result
shows how to deduce continuous  solutions of \eqref{prod} on $J$ when $c<d<0$ from those when $d>c>0$ and
vice versa.

\begin{pro}\label{p1}
	Suppose that $d>c>0$.
A map $g$ is a solution (resp. unique solution) of \eqref{prod} in $\mathcal{X}'\subseteq  \mathcal{C}(J,J)$ for $G\in \mathcal{X}\subseteq \mathcal{C}(J,J)$, where $\Psi_k \in \mathcal{Y}_k\subseteq  \mathcal{C}(J,\mathbb{R}_-)$ and $\psi_k \in \mathcal{Z}_k\subseteq  \mathcal{C}(J,J)$ for all $1\le k\le n$, if and only if $\tilde{g}=h^{-1}\circ g\circ h$ is a solution (resp. unique solution) of the equation
	\begin{eqnarray}\label{tildeG}
		\prod_{k=1}^{n}\tilde{\Psi}_k(g^k(\tilde{\psi}_k(x)))=\tilde{G}(x)
	\end{eqnarray}
	in $\mathcal{\tilde{X'}}\subseteq \mathcal{C}(\tilde{J},\tilde{J})$ for $\tilde{G} \in \mathcal{\tilde{X}}\subseteq  \mathcal{C}(\tilde{J},\tilde{J})$,
	where $h(x)=-x$, $\tilde{J}=-J:=\{-x:x\in J\}$, $\mathcal{\tilde{X}}=\{h^{-1}\circ g\circ h: g\in \mathcal{X}\} \subseteq \mathcal{C}(\tilde{J},\tilde{J})$, $\mathcal{\tilde{X}'}=\{h^{-1}\circ g\circ h: g\in \mathcal{X'}\}\subseteq \mathcal{C}(\tilde{J},\tilde{J})$, $\tilde{G}=h^{-1}\circ G\circ h$, and $\tilde{\Psi}_k=h^{-1}\circ \Psi_k\circ h \in \mathcal{\tilde{Y}}_k$,   $\tilde{\psi}_k=h^{-1}\circ \psi_k\circ h \in \mathcal{\tilde{Z}}_k$, $\mathcal{\tilde{Y}}_k=\{h^{-1}\circ g\circ h: g\in \mathcal{Y}_k\}\subseteq \mathcal{C}(\tilde{J},\mathbb{R}_+)$ and $\mathcal{\tilde{Z}}_k=\{h^{-1}\circ g\circ h: g\in \mathcal{Z}_k\}\subseteq \mathcal{C}(\tilde{J},\tilde{J})$  for all $1\le k\le n$.
\end{pro}

\begin{proof}
	Let  $g$ be a solution of \eqref{prod} in $ \mathcal{X'}$. Then clearly
 $\tilde{g}\in \mathcal{\tilde{X'}} \subseteq \mathcal{C}(\tilde{J},\tilde{J})$, and $\tilde{\Psi}_k\in \mathcal{\tilde{Y}}_k\subseteq \mathcal{C}(\tilde{J},\mathbb{R}_+)$
  and   $\tilde{\psi}_k\in \mathcal{\tilde{Z}}_k \subseteq \mathcal{C}(\tilde{J},\tilde{J})$ for all $1\le k\le n$. Also, for each $x\in \tilde{J}$ and $k\in \{1,2,\ldots, n\}$, we have $\tilde{G}(x)=-G(-x)$, $\tilde{g}^k(x)=-g^k(-x)$, $\tilde{\Psi}_k(x)=-\Psi_k(-x)$ and $\tilde{\psi}_k(x)=-\psi_k(-x)$.
 Therefore
 	\begin{eqnarray*}
		\prod_{k=1}^{n}\tilde{\Psi}_k(\tilde{g}^k(\tilde{\psi}_k(x)))=-\prod_{k=1}^{n}\Psi_k(g^k(\psi_k(-x)))=-G(-x)=\tilde{G}(x)
	\end{eqnarray*}
	for each $x\in \tilde{J}$, implying that $\tilde{g}$ is a solution of \eqref{tildeG} on $\tilde{J}$. The converse follows
	similarly. Now, in order
	to prove the uniqueness, assume that \eqref{prod} has a unique solution in $\mathcal{X'}$ and suppose that
	$\tilde{g}_1$, $\tilde{g}_2$ are any two solutions of \eqref{tildeG} in $\tilde{\mathcal{X'}}$. Then, by the ``if'' part of what we have proved
	above, there exist solutions $g_1$ and $g_2$ of \eqref{prod} in $\mathcal{X'}$ such that $\tilde{g}_1=h^{-1}\circ g_1\circ h$  and $\tilde{g}_2=h^{-1}\circ g_2\circ h$. By our assumption, we have $g_1=g_2$ and therefore $\tilde{g}_1=\tilde{g}_2$. The proof
	of the converse is similar.
\end{proof}


Although the above discussion is about continuous solutions, we can employ a similar approach of using logarithmic conjugacy to discuss smooth ($C^r$ for $r=1$ or larger) solutions of \eqref{prod} using those for \eqref{sum} given in \cite{Mu-Su-2007, Si-Zhang,zhang1990} whenever $0\notin J$.

{\bf Zero Problem:}
{\it
If $0\in J$, then the current approach of using	the logarithmic conjugacy, as used in  Theorem \ref{C-existance}, is not applicable for solving \eqref{prod} because $\log 0$ is not a well-defined real number}.
%
%

The Zero Problem occurs with continuity. It can be avoided if we discuss \eqref{prod} with weaker regularity (semi-continuity and integrability),
for which we need to use another method.


\section{Order-preserving solutions}\label{S2}

As indicated at the end of the above section, in order to deal with the Zero Problem,
we will consider solutions without continuity or with a `weak version' of continuity. As considering monotonicity in $\mathcal{G}(J;\delta,M)$ and $\mathcal{F}(J;\delta,M)$ just before (\ref{G-F}),
we start with orientation-preserving solutions, ignoring continuity.
We need the following preliminaries on complete lattices for our discussion in this section and the one that follow.


As defined in \cite{Szasz}, a relation $\preceq$ on a non-empty set $X$ is
called a {\it partial order} if it is
reflexive (i.e., $x\preceq x$ for all $x\in X$),
antisymmetric (i.e., $x=y$ whenever $x\preceq y$ and $y\preceq x$ in $X$),
and
transitive (i.e., $x\preceq z$ whenever  $x\preceq y$ and $y\preceq z$ in $X$).
$X$ endowed with a partial order $\preceq$
is called a {\it partially ordered set} (or simply a {\it poset}).
For a subset $E$ of the poset $X$,
$b\in X$ is called an {\it upper bound} (resp. a {\it lower bound}) of $E$
if $x\preceq b$ (resp. $b\preceq x$) for all $x\in E$.
Further, $b$ is called the
{\it least upper bound} or {\it supremum} (resp. {\it greatest lower bound} or {\it infimum}),
denoted by $\sup_X E$ (resp. $\inf_X E$),
if $b$ is an  upper bound (resp.  lower bound) of $E$ and
every upper bound (resp. lower bound) $z$ of $E$ satisfies
$b\preceq z$ (resp. $z\preceq b$).
A poset $X$ is called a {\it lattice} if $\sup_X\{x,y\}$, $\inf_X\{x,y\}\in X$ for every $x,y \in X$.
$X$ being a lattice in the partial order $\preceq$ is said to be {\bf (i)} {\it join-complete}
if $\sup_X E \in X$ for every non-empty subset $E$ of $X$; {\bf (ii)} {\it meet-complete} if $\inf_X E \in X$
for every non-empty subset $E$ of $X$; {\bf (iii)} {\it complete} if $X$ is both join- and meet-complete.
$X$ is said
to be {\it simply ordered} (or a {\it chain}) if at least one of the relations $x\preceq y$ and $y\preceq x$ hold whenever $x,y \in X$.
Further,
a non-empty
subset $E$ of $X$ is said to be {\bf (i)} a {\it sublattice} of $X$ if
$\sup_X\{x,y\}$, $\inf_X\{x,y\}\in E$ for every $x,y \in E$; {\bf (ii)}
{\it convex} if $\{z\in X: x\preceq z \preceq y\}\subseteq E$ whenever $x \preceq y$ in $E$;
{\bf (iii)} a {\it complete sublattice} of $X$ if
$\sup_X Y$ and $\inf_X Y$ exist, and both are in $E$ for every non-empty subset $Y$ of $E$.
For convenience, we use  $(X, \preceq)$ to denote a lattice $X$ in the partial order $\preceq$.

A map $g:X\to X'$, where $(X, \preceq)$ and $(X', \preceq')$ are lattices, is said to be {\it order-preserving} (resp. {\it strictly order-preserving}) if $g(x)\preceq' g(y)$ (resp. $g(x)\prec' g(y)$) in $X'$ whenever  $x\preceq y$ (resp. $x\prec y$) in $X$.
Let $\mathcal{G}(X,X')$  and $\mathcal{G}_{op}(X,X')$ denote the poset of all maps  and order-preserving maps of $X$ into $X'$
respectively in the {\it pointwise partial order} $\trianglelefteq$ defined by $g_1\trianglelefteq g_2$ if $g_1(x)\preceq' g_2(x)$ for all $x\in X$. For convenience, we denote $\mathcal{G}(X,X)$  (resp. $\mathcal{G}_{op}(X,X)$) by $\mathcal{G}(X)$  (resp. $\mathcal{G}_{op}(X)$).
As in \cite{Glazowska},
for $g_1, g_2\in \mathcal{G}(X)$, we say that $g_1$ {\it subcommutes} with $g_2$
if
$
g_1\circ g_2 \trianglelefteq g_2\circ g_1.
$

\begin{lm}\label{Lm1}{\rm(\cite{GZ2021})}
	The following assertions are true
for a lattice $(X,\preceq)$:
	\begin{description}
		\item[(i)] Both $\mathcal{G}(X)$ and $\mathcal{G}_{op}(X)$ are lattices in the partial order $\trianglelefteq$.
		
		\item[(ii)] If $X$ is a
		complete lattice, then $(\mathcal{G}_{op}(X), \trianglelefteq)$ is also a complete lattice.
		
		\item[(iii)] If $g \in \mathcal{G}_{op}(X)$, then $g^k\in \mathcal{G}_{op}(X)$ for each $k\in \mathbb{N}$.
		
		\item[(iv)] If $g_1, g_2 \in \mathcal{G}_{op}(X)$ such that $g_1\trianglelefteq g_2$, then
		$g_1^k\trianglelefteq g_2^k$ for each $k\in \mathbb{N}$.
		
		
		\item[(v)] If $g_1, g_2 \in \mathcal{G}_{op}(X)$ such that $g_1$ subcommutes with $g_2$ and $g_1(x)\preceq g_2(x)$, then $g_1^k(x)\preceq g_2^k(x)$ for each $k\in \mathbb{N}$.
	\end{description}
\end{lm}

\begin{lm}{\rm (Knaster-Tarski \cite{Knaster,Tarski})}\label{L0}
	Let $(X, \preceq)$ be a complete lattice and $g$ an order-preserving self-map on $X$.
	Then the set of all fixed points of  $g$ is a non-empty complete sublattice of $X$.
	Furthermore,
	$g$ has the minimum fixed point $x_*$ and the maximum fixed point $x^*$ in $X$ given by
	$x_*=\inf\{x\in X: g(x)\preceq x\}$ and
	$x^*=\sup\{x\in X: x\preceq g(x)\}$.
\end{lm}

The first part of this lemma can also be found in the expository article \cite{subra2000}.
A part of the second, showing that $\sup\{x\in X: x\preceq g(x)\}$ and $\inf\{x\in X: g(x)\preceq x\}$ are fixed points of $g$ thereby proving the existence of a fixed point, can also be found in the book \cite{Gratzer1978}.

Having the above preliminaries, we will now discuss order-preserving
solutions of \eqref{prod} on compact intervals in $\mathbb{R}$,
which will serve as tools for our subsequent discussion of
 semi-continuous solutions and integrable solutions in Sections \ref{S4} and \ref{sec6}, respectively.
Henceforth,
for the entirety of this section and the one that follows,
let $X$ be a compact interval $J:=[c,d]$ of $\mathbb{R}$ such that $d>\max\{0,c\}$, which is also a simply ordered complete
lattice in the usual order $\le$. For each $\delta>0$ such that $c\le\delta\le d$, let
\begin{align*}
	\mathcal{G}(J;\delta)&:=\{g\in \mathcal{G}(J):g(x)\ge \delta~\text{for all}~x\in J\},\\
	\mathcal{G}_{op}(J;\delta)&:=\{g\in \mathcal{G}_{op}(J):g(x)\ge \delta~\text{for all}~x\in J\}.
\end{align*}


\begin{thm}\label{Thm1}
Let $\delta>0$ such that $c\le\delta\le d$, and let $\lambda>0$, $\lambda_1\le 1$, $\lambda_k\le 0$ for $2\le k\le n$ such that $\sum_{k=1}^{n}\lambda_k=\lambda$.
Further, let
$\psi_1= {\rm id}$ on $J$ and
$\psi_k\in \mathcal{G}_{op}(J)$ for $2\le k\le n$,
and
let
$\Psi_1 \in \mathcal{G}(J, \mathbb{R})$ such that $\Psi_1(\cdot)=({\rm id}(\cdot))^{\lambda_1}$ on $[\delta,d]$,
and
$\Psi_k \in \mathcal{G}(J, \mathbb{R})$ such that $\Psi_k(\cdot)=(\Xi_k(\cdot))^{\lambda_k}$ for some $\Xi_k\in \mathcal{G}_{op}([\delta,d])$ for $2\le k\le n$.
Then
	the set $\mathcal{S}_{op}(J;\delta)$ of all solutions of  \eqref{prod} in $\mathcal{G}_{op}(J;\delta)$
	is a non-empty complete sublattice of $\mathcal{G}_{op}(J;\delta)$
if
$G\in \mathcal{G}_{op}(J;\delta)$ satisfies
$G(c)\ge \delta^\lambda$ and $ G(d)\le d^\lambda$.
Moreover,
	\eqref{prod} has the minimum solution $g_*$ and the maximum solution $g^*$ in $\mathcal{G}_{op}(J;\delta)$  given by
	\begin{eqnarray*}
		&&g_*=\inf\left\{g\in \mathcal{G}_{op}(J;\delta):   G\trianglelefteq  	\prod_{k=1}^{n}\Psi_k\circ g^k\circ \psi_k\right\},
		\\
		&&g^*=\sup\left\{g\in \mathcal{G}_{op}(J;\delta): \prod_{k=1}^{n}\Psi_k\circ g^k\circ \psi_k \trianglelefteq  G\right\}.
	\end{eqnarray*}
\end{thm}

\begin{proof}
	For each $H \in \mathcal{G}_{op}(J;\delta)$, we first see that solving \eqref{prod} for the map $G(x)= (H(x))^\lambda$ can be simplified to a fixed point problem.
	
	\noindent {\it Step 1.}
	Prove that $g$ is a solution of the equation
	\begin{eqnarray} \label{1}
		\prod_{k=1}^{n}\Psi_k(g^k(\psi_k(x)))=(H(x))^\lambda
	\end{eqnarray}
	in $\mathcal{G}(J;\delta)$ if and only if it is a fixed point
	of the operator $T:\mathcal{G}(J;\delta)\to \mathcal{G}(J;\delta)$ defined  by
	\begin{eqnarray}\label{T}
		Tg(x)=(g(x))^{\alpha_1}\cdot \left(\prod_{k=2}^{n}(\Xi_k(g^k(\psi_k(x))))^{\alpha_k}\right)\cdot(H(x))^\alpha
	\end{eqnarray}
	where 
	$\alpha=\lambda$, $\alpha_1=1-\lambda_1$ and $\alpha_k=-\lambda_k$ for $2\le k\le n$.

	Using the assumptions on $\lambda$ and $\lambda_k$'s, we have
	\begin{eqnarray}\label{alphak ineuality}
		\alpha>0,\quad \alpha_k\ge 0~~ \text{for}~~1\le k\le n, \quad\text{and}\quad  \sum_{k=1}^{n}\alpha_k+ \alpha=1.
	\end{eqnarray}
Also, it is clear from the assumptions on the maps $\Xi_k$ and $\psi_k$  that $T$ is a well-defined map of $\mathcal{G}(J;\delta)$ into $\mathcal{G}(J)$. Further, for each $g\in \mathcal{G}(J;\delta)$ and  $x\in J$, we have
	\begin{eqnarray*}\label{02}
		c\le \delta= \delta^{\sum_{k=1}^{n}\alpha_k+\alpha}\le (g(x))^{\alpha_1}\cdot \left(\prod_{k=2}^{n}(\Xi_k(g^k(\psi_k(x))))^{\alpha_k}\right)\cdot(H(x))^\alpha\le d^{\sum_{k=1}^{n}\alpha_k+\alpha}=d,
	\end{eqnarray*}
	i.e.,
	$c\le \delta \le Tg(x) \le d$, proving that $Tg \in \mathcal{G}(J;\delta)$ for each $g\in \mathcal{G}(J;\delta)$. Therefore
	$T$ is a self-map of $\mathcal{G}(J;\delta)$.

	Let $g$ be a solution of \eqref{1} in $\mathcal{G}(J;\delta)$.
	Then, by \eqref{T} we have
	\begin{align*}
		Tg(x)&=(g(x))^{1-\lambda_1}\cdot\left(\prod_{k=2}^{n}(\Xi_k(g^k(\psi_k(x))))^{-\lambda_k}\right)\cdot(H(x))^\lambda
		\\
		&= g(x)\cdot\left((g(x))^{\lambda_1}\cdot \prod_{k=2}^{n}(\Xi_k(g^k(\psi_k(x))))^{\lambda_k}\right)^{-1}\cdot(H(x))^\lambda \\
			&= g(x)\cdot\left(\prod_{k=1}^{n}\Psi_k(g^k(\psi_k(x)))\right)^{-1}\cdot(H(x))^\lambda \\
		&=g(x)\cdot ((H(x))^\lambda)^{-1}\cdot (H(x))^\lambda\\
		&=g(x)
	\end{align*}
	for each $x\in J$,	implying that $g$ is a fixed point of $T$. This proves the ``only if'' part. To prove the ``if'' part, suppose that $g$ is a fixed point of $T$ in $\mathcal{G}(J;\delta)$. Then
	\begin{align*}
	\prod_{k=1}^{n}\Psi_k(g^k(\psi_k(x)))
	&= \Psi_1(g(\psi_1(x))) \cdot \prod_{k=2}^{n}\Psi_k(g^k(\psi_k(x)))\\
		&=(g(x))^{\lambda_1}\cdot \prod_{k=2}^{n}(\Xi_k(g^k(\psi_k(x))))^{\lambda_k}\\
		&=(g(x))^{1-\alpha_1}\cdot\left(\prod_{k=2}^{n}(\Xi_k(g^k(\psi_k(x))))^{-\alpha_k}\right)
		\\
		&= g(x)\cdot\left((g(x))^{\alpha_1}\cdot \prod_{k=2}^{n}(\Xi_k(g^k(\psi_k(x))))^{\alpha_k}\cdot (H(x))^\alpha\right)^{-1}\cdot(H(x))^\alpha \\
		&=g(x)\cdot (g(x))^{-1}\cdot (H(x))^\alpha\\
		&=(H(x))^\lambda,
	\end{align*}
	implying that $g$ is a solution of \eqref{1}.

	We now prove in the following three steps that
	the set of all solutions of \eqref{1}  in $\mathcal{G}_{op}(J;\delta)$ is a non-empty complete sublattice of $\mathcal{G}_{op}(J;\delta)$.

	\noindent
	{\it Step 2.} Construct an order-preserving map $T:\mathcal{G}_{op}(J;\delta) \to \mathcal{G}_{op}(J;\delta)$.
	
	Define a map $T$ on  $\mathcal{G}_{op}(J;\delta)$ as in \eqref{T},
	where $\alpha$ and $\alpha_k$'s are chosen as in Step 1.
	Then, by using the assumptions on $\lambda$ and $\lambda_k$'s, we see that $\alpha$ and $\alpha_k$'s satisfy \eqref{alphak ineuality}.
	Further, by a similar argument as in Step 1, it follows that $Tg$ is a self-map on $J$ and $Tg(x)\ge \delta$ for all $x\in J$.
	
		
		Next, to prove that $Tg$ is order-preserving, consider any $x, y\in J$ such that $x\le  y$.  Since $H, g, \Xi_k$ and $\psi_k$ are order-preserving on their domains
		 for $1\le k\le n$, by using result {\bf (iii)} of Lemma \ref{Lm1} we have
		\begin{align*}\label{Tf op}
			Tg(x)&=(g(x))^{\alpha_1}\cdot \left(\prod_{k=2}^{n}(\Xi_k(g^k(\psi_k(x))))^{\alpha_k}\right)\cdot(H(x))^\alpha  \nonumber\\
			&\le (g(y))^{\alpha_1}\cdot \left(\prod_{k=2}^{n}(\Xi_k(g^k(\psi_k(y))))^{\alpha_k}\right)\cdot(H(y))^\alpha  \nonumber\\
			&=Tg(y).
		\end{align*}	
		Therefore $T$ is a self-map of $\mathcal{G}_{op}(J;\delta)$.
		
		Finally, to prove that $T$ is order-preserving,
		consider any $g_1, g_2 \in \mathcal{G}_{op}(J;\delta)$ such that $g_1\trianglelefteq g_2$. Then by result {\bf (iv)}
		of Lemma \ref{Lm1},
we have $g_1^k\trianglelefteq g_2^k$ for $1\le k \le n$.
Therefore, as $\Xi_k$ is order-preserving on $[\delta,d]$ for $1\le k\le n$, we have
		\begin{align*}
			Tg_1(x)&=(g_1(x))^{\alpha_1}\cdot\left(\prod_{k=2}^{n}(\Xi_k(g_1^k(\psi_k(x))))^{\alpha_k}\right)\cdot(H(x))^\alpha  \nonumber\\
			&\le(g_2(x))^{\alpha_1}\cdot \left(\prod_{k=2}^{n}(\Xi_k(g_2^k(\psi_k(x))))^{\alpha_k}\right)\cdot(H(x))^\alpha  \nonumber\\
			&=Tg_2(x)
		\end{align*}
		for each $x\in J$,  i.e., $Tg_1\trianglelefteq Tg_2$. Hence $T$ is order-preserving.
		
		\noindent
		{\it Step 3.} Prove that  $(\mathcal{G}_{op}(J;\delta), \trianglelefteq)$ is a complete lattice.
		
		Consider an arbitrary subset $\mathcal{E}$ of $\mathcal{G}_{op}(J;\delta)$. If $\mathcal{E}=\emptyset$, then the constant map $\phi:J\to J$ defined by $\phi(x)=d$
		is the infimum of $\mathcal{E}$ in $\mathcal{G}_{op}(J;\delta)$. If $\mathcal{E}\ne \emptyset$, then the map $\phi:J\to J$ defined by
		$\phi(x)=\inf\{g(x): g \in \mathcal{E}\}$ is the infimum
		of $\mathcal{E}$ in $\mathcal{G}_{op}(J;\delta)$.  Thus every subset of $\mathcal{G}_{op}(J;\delta)$ has the infimum in $\mathcal{G}_{op}(J;\delta)$. Therefore by Lemma $14$ of \cite{Gratzer1978},
		which says that if every subset of a poset $P$ has the infimum in $P$ then $P$ is complete,
		we get that $\mathcal{G}_{op}(J;\delta)$ is a complete lattice.
		

		\noindent
		{\it Step 4.} Prove that the set of all solutions of \eqref{1}
		in $\mathcal{G}_{op}(J;\delta)$ is a non-empty complete sublattice of $\mathcal{G}_{op}(J;\delta)$.
		
		From Step 1
		we see that $T$ is an order-preserving self-map of the lattice  $\mathcal{G}_{op}(J;\delta)$, which is complete by Step 3.
		Therefore by Lemma \ref{L0},
		the set of all fixed points of $T$ in $\mathcal{G}_{op}(J;\delta)$, and hence by Step 1, the set of all solutions of \eqref{1}  in $\mathcal{G}_{op}(J;\delta)$ is a non-empty complete sublattice of $\mathcal{G}_{op}(J;\delta)$.

		Now, in order to prove our result,	given $G$ as above, let $H(x):=(G(x))^{1/\lambda}$ for all $x\in J$. Then, since $\lambda>0$, clearly $H$ is order-preserving on $J$. Also, we have
		\begin{eqnarray}\label{G}
			\delta \le  (G(c))^\frac{1}{\lambda} \le  (G(x))^\frac{1}{\lambda} \le (G(d))^\frac{1}{\lambda} \le d,\quad \forall x\in J.
		\end{eqnarray}
		Therefore $H\in \mathcal{G}_{op}(J;\delta)$.  This implies by the above part that the
		set of all solutions of \eqref{1},
		and hence that of \eqref{prod}
		in $\mathcal{G}_{op}(J;\delta)$ is a non-empty complete sublattice of $\mathcal{G}_{op}(J;\delta)$.
		
		In particular, \eqref{prod} has the minimum solution $g_*$ and the maximum solution $g^*$ in $\mathcal{G}_{op}(J;\delta)$, which are in fact $\min \mathcal{S}_{op}(J;\delta)$ and $\max \mathcal{S}_{op}(J;\delta)$, respectively.
		Further, by Lemma \ref{L0}, we have $g_*=\inf\{g\in \mathcal{G}_{op}(J;\delta):  Tg \trianglelefteq  g\}$ and $g^*=\sup\{g\in \mathcal{G}_{op}(J;\delta): g\trianglelefteq  Tg\}$, where $H$ is the map defined by $H(x)=(G(x))^{1/\lambda}$ for all $x\in J$.
		This completes the proof.
	\end{proof}
	
	It is worth noting that the result in Step 1 of the above theorem is not true in general if $\Psi_1(\cdot)\ne ({\rm id}(\cdot))^{\lambda_1}$ on $[\delta,d]$
or $\psi_1\ne {\rm id}$ on $J$.
	Further, the reason why we do not assume all $\lambda_k$'s are positive in the above theorem
	is that in that case $Tg$ is not necessarily a self-map of $J$ whenever $g\in \mathcal{G}_{op}(J;\delta)$.
	The following result is devoted to uniqueness of solutions.

\begin{thm}\label{Thm3}
Let $\delta>0$ such that $c\le\delta\le d$. Further, let  $\psi_1=id$ on $J$ and $\psi_k \in \mathcal{G}_{op}(J)$ for $2\le k\le n$, and let $\Psi_1 \in \mathcal{G}(J, \mathbb{R}_+)$  is strictly order-preserving on $[\delta,d]$,  and $\Psi_k \in \mathcal{G}(J,\mathbb{R}_+)$   is  order-preserving on $[\delta,d]$ for $2\le k \le n$. Then 	the following assertions are true for $G \in \mathcal{G}_{op}(J;\delta)$.
		\begin{description}
			\item[(i)] If $g_1, g_2 \in \mathcal{G}_{op}(J;\delta)$ are solutions of \eqref{prod} on $J$ such that 	$g_1\trianglelefteq g_2$, then $g_1=g_2$.
			
			\item[(ii)] If $g_1, g_2 \in \mathcal{G}_{op}(J;\delta)$ are solutions of \eqref{prod} on $J$ such that $g_1\circ g_2=g_2\circ g_1$, then $g_1=g_2$.
		\end{description}
\end{thm}

	\begin{proof}
		Let  $g_1, g_2 \in \mathcal{G}_{op}(J;\delta)$ be solutions of \eqref{prod} on $J$ such that $g_1\trianglelefteq g_2$, and  suppose that $g_1\ne g_2$ on $J$. Then there exists $x\in J$ such that $g_1(x)< g_2(x)$, and by result {\bf (iv)} of Lemma \ref{Lm1}, we have $g_1^k\trianglelefteq g_2^k$, implying that $g_1^k(x)\le  g_2^k(x)$ for $2\le k\le n$. Therefore, by \eqref{prod}, we have
		\begin{eqnarray}\label{contr}
			G(x)=\prod_{k=1}^{n}\Psi_k(g_1^k(\psi_k(x)))<\prod_{k=1}^{n}\Psi_k(g_2^k(\psi_k(x)))=G(x),
		\end{eqnarray}
	which is a contradiction. Hence $g_1=g_2$ on $J$, proving result {\bf (i)}.

		In order to prove result {\bf (ii)},
		consider any solutions  $g_1, g_2 \in \mathcal{G}_{op}(J;\delta)$ of \eqref{prod} such that $g_1\circ g_2=g_2\circ g_1$, and
		suppose that $g_1\ne g_2$ on $J$. Then there exists $x\in J$ such that $g_1(x)\ne g_2(x)$, implying that either $g_1(x)< g_2(x)$ or $g_2(x)< g_1(x)$. If $g_1(x)< g_2(x)$, then by result {\bf (v)} of Lemma \ref{Lm1} we have $g_1^k(x)\le g_2^k(x)$ for $2\le k \le n$, and therefore  we arrive at \eqref{contr},
		 which is a contradiction. We get a similar contradiction if $g_2(x)< g_1(x)$. Hence $g_1=g_2$ on $J$.
	\end{proof}

	As seen in \cite{GZ2021}, the condition $g_1\trianglelefteq g_2$ and the condition $g_1\circ g_2=g_2\circ g_1$,
	assumed in results {\bf (i)} and {\bf (ii)} of Theorem \ref{Thm3} respectively, are independent.
	This shows
	that neither {\bf (i)} implies {\bf (ii)} nor {\bf (ii)} implies {\bf (i)} in Theorem \ref{Thm3}.

	
A map $g:X\to X'$, where $(X, \preceq)$ and $(X', \preceq')$ are lattices, is said to be {\it order-reversing} if $g(y)\preceq' g(x)$  in $X'$ whenever  $x\preceq y$ in $X$.
We
	remark that the current approach with the map $T$ defined in \eqref{T}, employed in Theorem \ref{Thm1},
	cannot be used to solve \eqref{prod} if
	$G\in \mathcal{G}_{or}(J;\delta)$, the complete lattice of all  order-reversing self-maps $g$ of $J$ with $g(x)\ge \delta$ for all $x\in J$ in the partial order $\trianglelefteq$.
	In fact,
	in the case that
	$G\in \mathcal{G}_{or}(J;\delta)$,
	assuming that  $\lambda:=\sum_{k=1}^{n}\lambda_k\ne0$,
	we see that $Tg$ is not necessarily order-preserving on $J$ for $g\in \mathcal{G}_{op}(J)$ no matter what
	the maps $\Xi_k, \psi_k$ and constants $\lambda_k$ are,
	because
	the function $x\mapsto (H(x))^\alpha$ in the product defining
	$Tg$
	is not order-preserving. Additionally, for a similar reason, the current approach with $T$ order-preserving cannot be used in general to seek a solution $g$ of \eqref{prod} in $\mathcal{G}_{or}(J;\delta)$ no matter whether $G$ is in $\mathcal{G}_{op}(J;\delta)$ or $\mathcal{G}_{or}(J;\delta)$.
	We do not consider the case that $\lambda=0$, where $G$ is not involved in $T$.

	Besides, a similar approach with $T$ order-reversing cannot be used to seek a solution $g$ of \eqref{prod} in $\mathcal{G}_{op}(J;\delta)$ or $\mathcal{G}_{or}(J;\delta)$
	no matter whether
	$\psi_k$ is in $\mathcal{G}_{op}(J)$ or $\mathcal{G}_{or}(J)$ (the complete lattice of all order-reversing self-maps of $J$ in the partial order $\trianglelefteq$),
 $\Xi_k$ is in $\mathcal{G}_{op}([\delta,d])$ or $ 	\mathcal{G}_{or}([\delta,d])$ (the complete lattice of all order-reversing self-maps of $[\delta,d]$ in the partial order $\trianglelefteq$), and	$G$ is in $\mathcal{G}_{op}(J;\delta)$ or $\mathcal{G}_{or}(J;\delta)$.
	In fact, Lemma \ref{L0} is not true if `order-preserving' is replaced with `order-reversing', as seen from the following example:
	Let $X$ be the complete lattice $\{x_1,x_2,x_3,x_4\}$ in the partial order $\preceq$ such that
	$x_1\preceq x_2\preceq x_4$ and $x_1\preceq x_3\preceq x_4$,
	and $g:X\to X$ be the order-reversing map such that
	$g(x_1)=x_4$, $g(x_2)=x_3$, $g(x_3)=x_2$ and $g(x_4)=x_1$. Then $g$ has no fixed points in $X$.
	For a similar reason, the approach of Theorem \ref{Thm2}  cannot be employed for other types of monotonicity.
	
	%
	%

\section{Semi-continuous solutions}
\label{S4}

The above section is
 devoted to order-preserving solutions, where considered nothing about continuity.
We now additionally consider semi-continuity
and give results on the existence and uniqueness of order-preserving semi-continuous solutions of \eqref{prod} on $J$.
As defined in \cite{Engelking1989},
 a map $g:J \to \mathbb{R}$ is said to be
{\it USC}, abbreviation of {\it upper semi-continuous},
(resp. {\it LSC}, abbreviation of {\it lower semi-continuous},)
at $x_0 \in J$ if
for every $\rho \in \mathbb{R}$ satisfying $g(x_0)<\rho$ (resp. $g(x_0)>\rho$) there exists a neighbourhood  $U$ of $x_0$ in $J$
such that $g(y)<\rho$ (resp. $g(y)>\rho$) for all $y\in U$. Equivalently, $g$ is USC (resp. LSC) at $x_0$ if
$\limsup_{x \to x_0} g(x) \le g(x_0)$ (resp. $\limsup_{x \to x_0} g(x) \ge g(x_0)$).  $g$ is said to be USC (resp. LSC) on $J$ if $g$ is USC (resp. LSC) at  each point of $J$.
Let
\begin{align*}
		\mathcal{G}^{usc}(J,\mathbb{R}):=&\{g\in \mathcal{G}(J,\mathbb{R}): g~\text{is USC on}~J\},\\
			\mathcal{G}^{lsc}(J,\mathbb{R}):=&\{g\in \mathcal{G}(J,\mathbb{R}): g~\text{is LSC on}~J\},\\
		\mathcal{G}^{usc}(J,\mathbb{R}_+):=&\{g\in \mathcal{G}(J,\mathbb{R}_+): g~\text{is USC on}~J\},\\
			\mathcal{G}^{lsc}(J,\mathbb{R}_+):=&\{g\in \mathcal{G}(J,\mathbb{R}_+): g~\text{is LSC on}~J\},\\
	\mathcal{G}^{usc}_{op}(J):=&\{g\in \mathcal{G}_{op}(J): g~\text{is USC on}~J\},\\
	\mathcal{G}^{lsc}_{op}(J):=&\{g\in \mathcal{G}_{op}(J): g~\text{is LSC on}~J\},\\
	\mathcal{G}^{sc}(J,\mathbb{R}_+)&:=\mathcal{G}^{usc}(J,\mathbb{R}_+)\cup \mathcal{G}^{lsc}(J,\mathbb{R}_+),\\
	\mathcal{G}^{sc}_{op}(J)&:=\mathcal{G}^{usc}_{op}(J)\cup \mathcal{G}^{lsc}_{op}(J),
\end{align*}
 and for each $\delta >0$ such that $c\le \delta\le d$, let
 \begin{align*}
 	\mathcal{G}_{op}^{usc}(J;\delta):=&\{g\in \mathcal{G}_{op}^{usc}(J):g(x)\ge \delta~\text{for all}~x\in J\},\\
 	\mathcal{G}_{op}^{lsc}(J;\delta):=&\{g\in \mathcal{G}_{op}^{lsc}(J):g(x)\ge \delta~\text{for all}~x\in J\},\\
 	\mathcal{G}^{sc}_{op}(J;\delta)&:=\mathcal{G}^{usc}_{op}(J;\delta)\cup \mathcal{G}^{lsc}_{op}(J;\delta).
 \end{align*}
By Theorem \ref{Thm1}, equation \eqref{prod} has a solution $g$ in $\mathcal{G}_{op}(J;\delta)$ for each $G\in  \mathcal{G}^{sc}_{op}(J;\delta)$
satisfying $G(c)\ge \delta^\lambda$ and $ G(d)\le d^\lambda$;
however we cannot conclude that $g$ is semi-continuous because $\mathcal{G}^{sc}_{op}(J;\delta)\subsetneq \mathcal{G}_{op}(J;\delta)$. For semi-continuous solutions
we have the following.

\begin{thm}\label{Thm2}
	Let $\delta>0$ such that $c\le\delta\le d$, and let $\lambda>0$, $\lambda_1\le 1$, $\lambda_k\le 0$ for $2\le k\le n$ such that $\sum_{k=1}^{n}\lambda_k=\lambda$.  Further, let
	 $\psi_1= {\rm id}$ on $J$ and $\psi_k\in \mathcal{G}^{usc}_{op}(J)$ for $2\le k\le n$, and let
	$\Psi_1 \in \mathcal{G}^{usc}(J, \mathbb{R})$ with $\Psi_1(\cdot)=({\rm id}(\cdot))^{\lambda_1}$ on $[\delta,d]$,  and
	$\Psi_k \in \mathcal{G}^{usc}(J, \mathbb{R}_+)$ with $\Psi_k(\cdot)=(\Xi_k(\cdot))^{\lambda_k}$ for some $\Xi_k\in \mathcal{G}^{usc}_{op}(J;\delta)$ for $2\le k\le n$.
Then the set $\mathcal{S}_{op}^{usc}(J;\delta)$ of all solutions of equation \eqref{prod} in $\mathcal{G}_{op}^{usc}(J;\delta)$
	is a non-empty complete sublattice of $\mathcal{G}_{op}^{usc}(J;\delta)$ 	if  $G\in \mathcal{G}_{op}^{usc}(J;\delta)$ satisfies $G(c) \ge \delta^\lambda$ and $G(d)\le d^\lambda$.
Moreover,
	\eqref{prod} has the minimum solution $g_*$ and the maximum solution $g^*$ in $\mathcal{G}_{op}^{usc}(J;\delta)$  given by
\begin{eqnarray*}
	&&g_*=\inf\left\{g\in \mathcal{G}_{op}^{usc}(J;\delta):   G\trianglelefteq  	\prod_{k=1}^{n}\Psi_k\circ g^k\circ \psi_k\right\},
	\\
	&&g^*=\sup\left\{g\in \mathcal{G}_{op}^{usc}(J;\delta): \prod_{k=1}^{n}\Psi_k\circ g^k\circ \psi_k \trianglelefteq  G\right\}.
\end{eqnarray*}
Moreover, these results are also true when upper semi-continuity
is replaced with lower semi-continuity.
\end{thm}

\begin{proof}
		Let $G\in \mathcal{G}^{usc}_{op}(J;\delta)$ be arbitrary.
		
			\noindent {\it Step 1.} Construct an order-preserving map $T:\mathcal{G}^{usc}_{op}(J;\delta) \to \mathcal{G}^{usc}_{op}(J;\delta)$.

		Given $\lambda$ and $\lambda_k$'s as above, let $H(x):=(G(x))^{1/\lambda}$ for all $x\in J$, and define a map $T$ on  $\mathcal{G}^{usc}_{op}(J;\delta)$ as in \eqref{T},
		where $\alpha$ and $\alpha_k$'s are chosen as in Step 1 of Theorem \ref{Thm1}. 	Then, by using the assumptions on $\lambda$ and $\lambda_k$'s, we see that  $\alpha$ and $\alpha_k$'s satisfy \eqref{alphak ineuality}. Also, since $\lambda>0$, clearly $H$ is an order-preserving USC map on $J$. Further, since $G$ is an order-preserving self-map of  $J$, we see that \eqref{G} is satisfied. 	
		Therefore $H\in \mathcal{G}^{usc}_{op}(J;\delta)$.
		
			Consider an arbitrary $g\in \mathcal{G}^{usc}_{op}(J;\delta)$.
			Since $\Xi_k,  \psi_k$ are order-preserving USC self-maps on $J$, so is $\Psi_k\circ g^k\circ \psi_k$ for $2\le k \le n$. Therefore $Tg$ is USC on $J$, being the product of non-negative USC maps $x\mapsto (g(x))^{\alpha_1}$, $x\mapsto (\Xi_k\circ g^k\circ \psi_k(x))^{\alpha_k}$ for $2\le k\le n$ and $x\mapsto (H(x))^\alpha$.
			Also, since $\mathcal{G}^{usc}_{op}(J;\delta)\subseteq \mathcal{G}_{op}(J;\delta)$, by using Step 1
of the proof of Theorem \ref{Thm1}, we have $Tg \in \mathcal{G}_{op}(J;\delta)$ and $T$ is order-preserving.
		 Therefore $T$ is an order-preserving self-map on $\mathcal{G}^{usc}_{op}(J;\delta)$.

				
\noindent
{\it Step 2.} Prove that  $(\mathcal{G}^{usc}_{op}(J;\delta), \trianglelefteq)$ is a complete lattice.
		
		Consider an arbitrary subset $\mathcal{E}$ of $\mathcal{G}^{usc}_{op}(J;\delta)$. If $\mathcal{E}=\emptyset$, then the constant map $\phi:J\to J$ defined by $\phi(x)=d$
is the infimum of $\mathcal{E}$ in $\mathcal{G}^{usc}_{op}(J;\delta)$. If $\mathcal{E}\ne \emptyset$, then the map $\phi:J\to J$ defined by
$\phi(x)=\inf\{g(x): g \in \mathcal{E}\}$ is the infimum
of $\mathcal{E}$ in $\mathcal{G}^{usc}_{op}(J;\delta)$.  Thus every subset of $\mathcal{G}^{usc}_{op}(J;\delta)$ has the infimum in $\mathcal{G}^{usc}_{op}(J;\delta)$. Therefore by Lemma $14$ of \cite{Gratzer1978},
which says that if every subset of a poset $P$ has the infimum in $P$ then $P$ is complete,
we know that $\mathcal{G}^{usc}_{op}(J;\delta)$ is a complete lattice.


\noindent
{\it Step 3.} Prove that	
$\mathcal{S}^{usc}_{op}(J;\delta)$ is a non-empty complete sublattice of $\mathcal{G}^{usc}_{op}(J;\delta)$.

From Steps 1 and 2, we see that $T$ is an order-preserving self-map of the complete lattice  $\mathcal{G}^{usc}_{op}(J;\delta)$. Hence, by Lemma \ref{L0}, the set of all fixed points of $T$ in $\mathcal{G}^{usc}_{op}(J;\delta)$ is a non-empty complete sublattice of $\mathcal{G}^{usc}_{op}(J;\delta)$. This implies by Step 1 of Theorem \ref{Thm1} that the set of all solutions of  \eqref{1}, and hence that of \eqref{prod} in $\mathcal{G}^{usc}_{op}(J;\delta)$ is a non-empty complete sublattice of $\mathcal{G}^{usc}_{op}(J;\delta)$,  because $H=G^{1/\lambda}$.
That is, $\mathcal{S}^{usc}_{op}(J;\delta)$
is a non-empty complete sublattice of $\mathcal{G}^{usc}_{op}(J;\delta)$.

In particular, \eqref{prod} has the minimum solution $g_*$ and the maximum solution $g^*$ in $\mathcal{G}^{usc}_{op}(J;\delta)$, which are in fact $\min \mathcal{S}^{usc}_{op}(J;\delta)$ and $\max \mathcal{S}^{usc}_{op}(J;\delta)$, respectively.
		 Further, by Lemma \ref{L0}, we have $g_*=\inf\{g\in \mathcal{G}_{op}^{usc}(J;\delta):  Tg \trianglelefteq  g\}$ and $g^*=\sup\{g\in \mathcal{G}_{op}^{usc}(J;\delta): g\trianglelefteq  Tg\}$.
		This completes the proof.
\end{proof}

We have the following results on uniqueness of solutions.

\begin{coro}\label{Thm5}
		Let $\delta>0$ such that $c\le\delta\le d$. Further, let  $\psi_1=id$ on $J$ and  $\psi_k \in \mathcal{G}^{usc}_{op}(J)$ for $2\le k\le n$, and let $\Psi_1 \in \mathcal{G}^{usc}(J, \mathbb{R}_+)$  is strictly order-preserving on $[\delta,d]$,  and $\Psi_k \in \mathcal{G}^{usc}(J,\mathbb{R}_+)$   is  order-preserving on $[\delta,d]$ for $2\le k \le n$. Then 	the following assertions are true for $G \in \mathcal{G}^{usc}_{op}(J;\delta)$.
\begin{description}
	\item[(i)] If $g_1, g_2 \in \mathcal{G}_{op}^{usc}(J;\delta)$ are solutions of \eqref{prod} on $J$ such that 	$g_1\trianglelefteq g_2$, then $g_1=g_2$.

	\item[(ii)] If $g_1, g_2 \in \mathcal{G}_{op}^{usc}(J;\delta)$ are solutions of \eqref{prod} on $J$ such that $g_1\circ g_2=g_2\circ g_1$, then $g_1=g_2$.
\end{description}
Further, these results are also true when upper semi-continuity
is replaced with lower semi-continuity.
\end{coro}

\begin{proof}
	Follows from Theorem \ref{Thm3}, since $\mathcal{G}_{op}^{sc}(J;\delta) \subseteq \mathcal{G}_{op}(J;\delta)$, $\mathcal{G}^{sc}(J, \mathbb{R}_+)\subseteq \mathcal{G}(J, \mathbb{R}_+)$ and  $\mathcal{G}^{sc}_{op}(J)\subseteq \mathcal{G}_{op}(J)$.
\end{proof}

	It is important to note that the current approach employed in Theorem \ref{Thm2}
cannot be used to solve \eqref{prod} for continuous solutions on $J$. Indeed, it is easy to see that $\mathcal{G}_{op}^c(J;\delta):=\{g\in \mathcal{G}_{op}(J;\delta):g~\text{is continuous on}~J\}$ is not a complete lattice in the partial order $\trianglelefteq$.

\section{Examples and remarks}\label{sec6}

The following examples illustrate our main theorems.

\begin{exmp}\label{Exmp1}
	{\rm		
Consider the functional equation
		\begin{eqnarray}\label{Ex3}
	(g(x))^\frac{4}{5}  \left(e^{(\log g^2(x))^2}\right)^\frac{1}{5}=\sqrt{x}\cdot  e^{\frac{(\log x)^2}{2}}
		\end{eqnarray}
		in the form \eqref{prod} on $J=[1,e]$, where $\lambda_1=4/5$, $\lambda_2=1/5$,  $G(x)=\sqrt{x}\cdot e^{(\log x)^2/2}$, $\psi_1(x)=\psi_2(x)=x$, $\Psi_1(x)=(\Xi_1(x))^{\lambda_1}$ and $\Psi_2(x)=(\Xi_2(x))^{\lambda_2}$ such that $\Xi_1(x)=x$ and $\Xi_2(x)=e^{(\log x)^2}$.
Let $I:=\log J$, and $f(x):=\log g(e^x)$, $F(x):=\log G(e^x)$, $\Upsilon_k(x):=\log \Xi_k(e^x)$, $\Phi_k(x):=\log \Psi_k(e^x)$ and $\phi_k(x):=\log\psi_k(e^x)$  for all $x\in I$ and $k=1,2$. Then \eqref{Ex3} reduces to the  equation
\begin{eqnarray*}
	\frac{4}{5}	f(x)+\frac{1}{5}(f^2(x))^2=\frac{x^2+x}{2},
\end{eqnarray*}
of the form \eqref{sum} on $I=[0,1]$, where $F(x)=(x^2+x)/2$, $\phi_1(x)=\phi_2(x)=x$, $\Phi_1(x)=\lambda_1\Upsilon_1(x)$ and $\Phi_2(x)=\lambda_2\Upsilon_2(x)$ such that $\Upsilon_1(x)=x$ and $\Upsilon_2(x)=x^2$  for all $x\in I$.
Note that  $F\in \mathcal{F}(I;1/2,3/2)\subseteq  \mathcal{F}(I;12/25,16/5) =\mathcal{F}(I;K_1\delta, K_0M)$, where $\delta=1/5$, $M=4$,  $K_0=\lambda_1l_1+\lambda_2l_2\delta=4/5$ and $K_1=\lambda_1L_1+\lambda_2L_2M=12/5$ with $l_1=1$, $l_2=0$, $L_1=1$ and $L_2=2$.
 This implies that $G\in \mathcal{G}(J;K_1\delta, K_0M)$.  Similarly, we have $\Upsilon_1\in \mathcal{F}(I;1,1)=\mathcal{F}(I;l_1,L_1)$ and $\Upsilon_2\in \mathcal{F}(I;0,2)=\mathcal{F}(I,l_2,L_2)$, implying that $\Xi_k\in \mathcal{G}(J;l_k,L_k)$ for $k=1,2$.
 Further, $K=\lambda_1l_1-(\lambda_2L_2-\lambda_2l_2\delta)=2/5>0$, implying (\ref{KKK}).
Thus, all the hypotheses of Theorems \ref{C-existance} are satisfied. Hence \eqref{Ex3} has a unique solution $g$ in $\mathcal{G}(J;1/5,4)$ that depends continuously on $G$.}
\end{exmp}

\begin{exmp}\label{E1}
{\rm	Consider the functional equation
	\begin{eqnarray}\label{Ex1}
	(g(x))^{\frac{4}{5}} \left(\frac{(g^2(x^3))^4+1}{3}\right)^{-\frac{3}{10}}
=\frac{x^2+1}{2}
	\end{eqnarray}
on $[0,1]$,	which is equation \eqref{prod} with $\lambda_1=4/5$, $\lambda_2=-3/10$, $G(x)=(x^2+1)/2$,  $\psi_1(x)=x$,   $\psi_2(x)=x^3$, $\Psi_1(x)=x^{\lambda_1}$ and $\Psi_2(x)=(\Xi_2(x))^{\lambda_2}$  such that   $\Xi_2(x)=(x^4+1)/3$.
 Clearly, $\lambda_1<1$,  $\lambda_2<0$, $\lambda=\lambda_1+\lambda_2=1/2>0$, $(G(0))^{1/\lambda}=1/4>1/5=:\delta$ and $(F(1))^{1/\lambda}=1$.
Also, it is easy to see that $\Psi_1\in \mathcal{G}^{usc}([0,1],\mathbb{R})$, $\Xi_2, G\in \mathcal{G}^{usc}_{op}([0,1],\delta)$, $\Psi_2\in \mathcal{G}^{usc}([0,1],\mathbb{R}_+)$ and $\psi_1, \psi_2 \in \mathcal{G}_{op}^{usc}([0,1])$.
 Thus, all the
hypotheses of Theorem \ref{Thm2} are satisfied. Hence \eqref{Ex1} has a solution in $\mathcal{G}_{op}^{usc}([0,1];\delta)$.
}
\end{exmp}

As defined in \cite{royden1988}, a map $g:J\to \mathbb{R}$ is called {\it Lebesgue measurable} (or simply {\it measurable})
if $\{x\in J: g(x)<\rho\}$ is Lebesgue measurable for each $\rho\in \mathbb{R}$.
A measurable function $g:J\to \mathbb{R}$ is said to be $L^p$ integrable (or simply $L^p$),
where $1\le p<\infty$,
if $|g|^p$ is Lebesgue integrable, i.e.,  $\int_{c}^{d}|g|^p d\mu<\infty$.
Although the  discussion in Section \ref{S2} is devoted to order-preserving solutions,
these results on order-preserving solutions also provide those for $L^p$ solutions on $J$ for $1\le p<\infty$, which have a weaker regularity,
since every order-preserving map on a compact interval is measurable and every bounded measurable map on a measurable set of finite measure is integrable by Proposition 3 of \cite[p.79]{royden1988}.
Further, it is worth noting that, in the special case $p=1$, these results indeed provide sufficient conditions for the existence and uniqueness of Riemann integrable solutions of \eqref{prod} for Riemann integrable maps $G$, because every Riemann integrable map is $L^1$ integrable and conversely, whenever the map is order-preserving.

\begin{exmp}
{\rm		The functional equation
	\begin{eqnarray}\label{Ex2}
	(g(x))^{\frac{3}{5}}\left(\frac{(g^2(\sin(\frac{\pi x}{2})))^4+2}{7}\right)^{-\frac{1}{10}}
	=G(x)
	\end{eqnarray}
	on $[0,1]$ with the function
\begin{eqnarray*}
	G(x)=\left\{\begin{array}{cll}
		\frac{x^4+1}{3}&\text{if}&0\le x<\frac{1}{2},\\
		\frac{x^3+1}{2}&\text{if}&\frac{1}{2}\le x\le 1
	\end{array}\right.
\end{eqnarray*}
 is of the form \eqref{prod}, where  $\lambda_1=3/5$, $\lambda_2=-1/10$,  $\psi_1(x)=x$,   $\psi_2(x)=\sin(\pi x/2)$, $\Psi_1(x)=x^{\lambda_1}$ and $\Psi_2(x)=(\Xi_2(x))^{\lambda_2}$ such that $\Xi_2(x)=(x^4+2)/7$. Then $\lambda_1<1$,   $\lambda_2<0$,  $\lambda=\lambda_1+\lambda_2=1/2>0$, $(G(0))^{1/\lambda}=1/9>1/10=:\delta$ and $(G(1))^{1/\lambda}=1$. Further, we have $\Psi_1\in \mathcal{G}([0,1],\mathbb{R})$, $\Xi_2, G\in \mathcal{G}_{op}([0,1],\delta)$, $\Psi_2\in \mathcal{G}([0,1],\mathbb{R})$ and $\psi_1, \psi_2 \in \mathcal{G}_{op}([0,1])$.
	Thus, all the
	hypotheses of Theorem \ref{Thm1} are satisfied. Hence \eqref{Ex2} has a solution in $\mathcal{G}_{op}([0,1];\delta)$. Although $G$ is not continuous on $[0,1]$, we see from our discussion in the preceding paragraph
	  that $G$ is $L^p$ integrable on $[0,1]$ for $1\le p<\infty$. For a similar reason, $g$ is also  $L^p$ integrable on $[0,1]$ for $1\le p<\infty$.}
\end{exmp}

Similarly to upper semi-continuity considered in Corollary \ref{Thm5}, we can also consider another property $P$,
which can be, for instance, continuity, differentiability or measurability. More precisely, Corollary \ref{Thm5} is indeed true whenever $\mathcal{G}^{usc}_{op}(J;\delta)$, $\mathcal{G}^{usc}(J, \mathbb{R}_+)$  and $\mathcal{G}^{usc}_{op}(J)$ are replaced by $\mathcal{G}^{P}_{op}(J;\delta)$, $\mathcal{G}^{P}(J, \mathbb{R}_+)$  and $\mathcal{G}^{P}_{op}(J)$, respectively,
since
 $\mathcal{G}_{op}^{P}(J;\delta) \subseteq \mathcal{G}_{op}(J;\delta)$, $\mathcal{G}^{P}(J, \mathbb{R}_+)\subseteq \mathcal{G}(J, \mathbb{R}_+)$ and  $\mathcal{G}^{P}_{op}(J)\subseteq \mathcal{G}_{op}(J)$,
 where
 \begin{align*}
 	\mathcal{G}^{P}(J,\mathbb{R}_+)&:=\{g\in \mathcal{G}(J,\mathbb{R}_+): g~\text{satisfies property P on}~J\},\\
 	\mathcal{G}^{P}_{op}(J)&:=\{g\in \mathcal{G}_{op}(J): g~\text{satisfies property P on}~J\},\\
 	\mathcal{G}^{P}_{op}(J;\delta)&:=\{g\in \mathcal{G}_{op}(J;\delta): g~\text{satisfies property P on}~J\}
 \end{align*}
for each $\delta>0$ such that $\max\{0,c\}<\delta< d$.
It is worth noting that, because the hypotheses of Theorem \ref{C-existance} and Corollary \ref{Thm5} are different, the results obtained from an analogue of Corollary \ref{Thm5}  with $P$ being the property of continuity provide additional results on the uniqueness of solutions of \eqref{prod} on $J$ in addition to those given in Theorem \ref{Thm5}.
However, because $\mathcal{G}_{op}^P(J;\delta)$ is not a complete lattice for this $P$, as noted at the end of Section \ref{S4}, we cannot use the approach used to prove Theorem \ref{Thm2} to give results on the existence of continuous solutions of \eqref{prod} on $J$.


Furthermore, because it is assumed that $0 <\lambda_1 \le 1$ in all of our main results, we cannot use them to solve the iterative root problem $g^n=G$ on $J$.
Moreover, as remarked in Section \ref{S2},
our current approach in proving Theorems \ref{Thm1} and \ref{Thm2} is not applicable to order-reversing cases. We leave these problems open for further investigation.

Besides, as seen in Section \ref{sec-2}, by using Proposition \ref{p1} and Theorem \ref{C-existance}, we can indeed solve \eqref{prod} for continuous solutions  on $J$ whenever $c<d<0$ and $\lambda_k\in \mathbb{Z}$ for all $1\le k\le n$ such that $\sum_{k=1}^{n}\lambda_k$ is odd.
On the other hand, if $c<d<0$ and $\lambda_k\in \mathbb{R}\setminus \mathbb{Z}$ for some $1\le k \le n$,  then (using the notation of Theorem \ref{C-existance})  for any maps $g, G, \Xi_k, \psi_k\in \mathcal{C}(J,J)$ the map $x\mapsto (g(x))^{\lambda_1}\cdot \prod_{k=2}^{n}(\Xi_k(g^k(\psi_k(x))))^{\lambda_k}$ is generally a multi-valued complex map, whereas the map $x\mapsto G(x)$ is a single valued real map. So, to obtain the equality
\eqref{prod},
we have to choose branches of the complex logarithm suitably, which depends not only on
 $x$ and $(g(x))^{\lambda_1}$ but also on each term of the product $\prod_{k=2}^{n}(\Xi_k(g^k(\psi_k(x))))^{\lambda_k}$.
 Therefore, solving \eqref{prod} for continuous solutions on $J$ in this case is very difficult.
Furthermore, as indicated in the Zero problem at the end of Section \ref{sec-2},  the current approach of using	the logarithmic conjugacy, as used in  Theorem \ref{C-existance}, is not applicable for solving \eqref{prod} on $J$ whenever $0\in J$, because $\log 0$ is not a well-defined real number. We leave this case open for further investigation.

Additionally, we remind that throughout our discussion in Sections \ref{S2} and \ref{S4} we assumed that $\delta \in [c,d]$ satisfies that $\delta>0$. However, if  $\delta\le 0$, which assumption is possible only when $c\le 0$,  then the result in Step 1 of
Theorem \ref{Thm1} is not true in general, and therefore our current approach of Theorems \ref{Thm1} and \ref{Thm2}  cannot be employed for solving \eqref{prod} on $J$.
Moreover, we have also assumed throughout the discussion in Sections \ref{S2} and \ref{S4} that $d>0$. Indeed, if $d=0$, then
 the current approach of Theorems \ref{Thm1} and \ref{Thm2}
 cannot be used to solve
 \eqref{prod} on $J$ because in this case the result in Step 1 of Theorem \ref{Thm1} is not true in general.
On the other hand,
if $d<0$, then solving \eqref{prod} on $J$ in this case is very difficult for the same reason as mentioned in the preceding paragraph,  which involved the choice of logarithm branches.



\noindent {\bf Acknowledgment:}
The author Chaitanya Gopalakrishna is supported
by the INSPIRE Faculty Fellowship of the Department of
Science and Technology, India through DST/INSPIRE/04/2022/001760.
The author Weinian Zhang is supported by NSFC grants \# 11831012 and \# 12171336 and National Key R\&D Program of China 2022YFA1005900.


{\small

}

%
%
%
%
%
%
%


\end{document}